\documentclass[letterpaper, 10 pt, conference]{ieeeconf}  

\IEEEoverridecommandlockouts 

\usepackage{xcolor}
\usepackage{float}
\usepackage{pgfplotstable}
\usepackage{pgfplots}
\pgfplotsset{
	table/search path={plot_figures},
}
\usepackage{times} 
\usepackage{amsmath} 
\usepackage{amssymb}
\usepackage{amsfonts}  
\usepackage{latexsym}
\usepackage{theorem}
\usepackage{epsfig}
\usepackage{graphicx} 
\usepackage{dsfont}
\usepackage{mathtools}
\usepackage{subfigure}
\usepackage{lscape}
\usepackage{color}
\usepackage{enumerate}
\usepackage[percent]{overpic}
\usepackage{hyperref}
\usepackage[export]{adjustbox}
\usepackage{tikz}
\usepackage{xr}
\usepackage{booktabs}
\usepackage{multirow}
\usetikzlibrary{arrows,%
	petri,%
	topaths}%
\usetikzlibrary{graphs,graphs.standard}
\usepackage{tkz-berge}
\usepackage[tableposition=top]{caption}
\usepackage{verbatim}
\usepackage{algorithm,algorithmic}
\pgfdeclarelayer{bg}    
\pgfsetlayers{bg,main}  

\newtheorem{theorem}{Theorem}

{\theorembodyfont{\rmfamily} }

\allowdisplaybreaks

\def\cu#1{{\color{black}#1}}  
\def\pd#1{{\color{black}#1}} 
\def\dd#1{{\color{black}#1}} 
\DeclareMathOperator*{\argmin}{arg\,min}

\DeclareMathOperator{\spn}{span}

\title{\LARGE \bf Distributed Computation of Wasserstein Barycenters over Networks}
\usepackage{authblk}
\graphicspath{{Figures/}}

\author{C\'{e}sar A.\ Uribe, Darina Dvinskikh, Pavel Dvurechensky, Alexander Gasnikov and Angelia Nedi\'{c}
	\thanks{C.A. Uribe (\textit{cauribe@mit.edu}) is with the Laboratory for Information and Decision Systems (LIDS), Massachusetts Institute of Technology. \dd{D. Dvinskikh
			(\textit{darina.dvinskikh@wias-berlin.de}) is with the Weierstrass Institute for Applied Analysis and Stochastics, and Institute for Information
			Transmission Problems RAS.}
		\pd{P. Dvurechensky (\textit{pavel.dvurechensky@wias-berlin.de}) is with Weierstrass Institute for Applied Analysis and Stochastics, and Institute for Information
			Transmission Problems RAS.}
		A. Gasnikov (\textit{gasnikov@yandex.ru}) is with the Moscow Institute of Physics and Technology, and the    Institute for Information Transmission Problems RAS. 
		A. Nedi\'{c} (\textit{angelia.nedich@asu.edu}) is with the ECEE Department, Arizona State University, and Moscow Institute of Physics and Technology. 
		The work of A. Gasnikov in Section III-C was supported by the grant of the president of Russian Federation no. MD-1320.2018.1. 
		The work of A. Nedi\'{c} and C.A.\ Uribe is supported by the National Science Foundation under grant no.\ CPS~15-44953.
		\pd{The work of P. Dvurechensky in Section III-C was supported by the grant of the president of Russian Federation no. MK-1806.2017.9.
			The work of A. Gasnikov and P. Dvurechensky in Sections III-A and III-B was conducted in IITP RAS and supported by the Russian Science Foundation
			grant (project 14-50-00150).
		}
	}  
}
\date{}

\begin{document}
	
	\newcommand{\cfbox}[2]{%
		\colorlet{currentcolor}{.}%
		{\color{#1}%
			\fbox{\color{currentcolor}#2}}%
	}

	\maketitle
	\begin{abstract}
		We propose a new \cu{class-optimal} algorithm for the distributed computation of Wasserstein Barycenters over networks. Assuming that each node in a graph has a probability distribution, we prove that every node can reach the barycenter of all distributions held in the network by using local interactions compliant with the topology of the graph. We provide an estimate for the minimum number of communication rounds required for the proposed method to achieve arbitrary relative precision both in the optimality of the solution and the consensus among all agents for undirected fixed networks. 
	\end{abstract}
	
	\section{Introduction}
	
	Optimal Transport (OT) distances (also known as \textit{earth mover's distances} or \textit{Wasserstein distances}) design an optimal plan to move ``mass" from one probability distribution to another. This problem can be traced back to the early work of Monge~\cite{Monge1781} and Kantorovich~\cite{Kantorovich1942} and has been of constant interest for allowing natural formulations to the problems of comparing, interpolating, and measuring distances of functions~\cite{Levy2017}. On the other hand, computational OT has gained popularity for its applications in learning theory~\cite{Frogner2015}, computer vision~\cite{Rabin2011}, computer graphics~\cite{Solomon2015}, statistical inference~\cite{Srivastava2015a}, information fusion~\cite{Bishop2014a}; and its relative complexity advantages with respect to classical methods~\cite{Dvurechensky2018}. \cu{Particularly, \textit{large-scale} OT has been of recent interest for applications where large quantities of data are available and efficient algorithms are required~\cite{Blondel2017,Seguy2017,Aude2016}}. Comprehensive accounts of the OT problem and its computational aspects can be found in~\cite{Villani2008,Solomon2017,Peyre2017,Levy2017}. 
	
	One of the common uses of the Wasserstein distance is the aggregation of distributions by considering their barycenter~\cite{Agueh2011}, which itself is another distribution~\cite{Cuturi2014}. Wasserstein Barycenters (WB) have been shown superior to traditional Euclidean means in a range of application such as image processing~\cite{Agueh2011}, economics and finance~\cite{Beiglbock2013}, and condensed matter physics~\cite{Buttazzo2012}. Figure~\ref{fig:sevens} shows a sample of $30$ images of the digit $7$ from the MNIST dataset~\cite{LeCun1998}, and their respective Euclidean mean and Wasserstein mean. The WB better captures the structural features of the input images.
	
	\begin{figure}[tb!]
		\centering
			\includegraphics[trim={0 0 0 6.2cm},clip,width=0.40\textwidth]{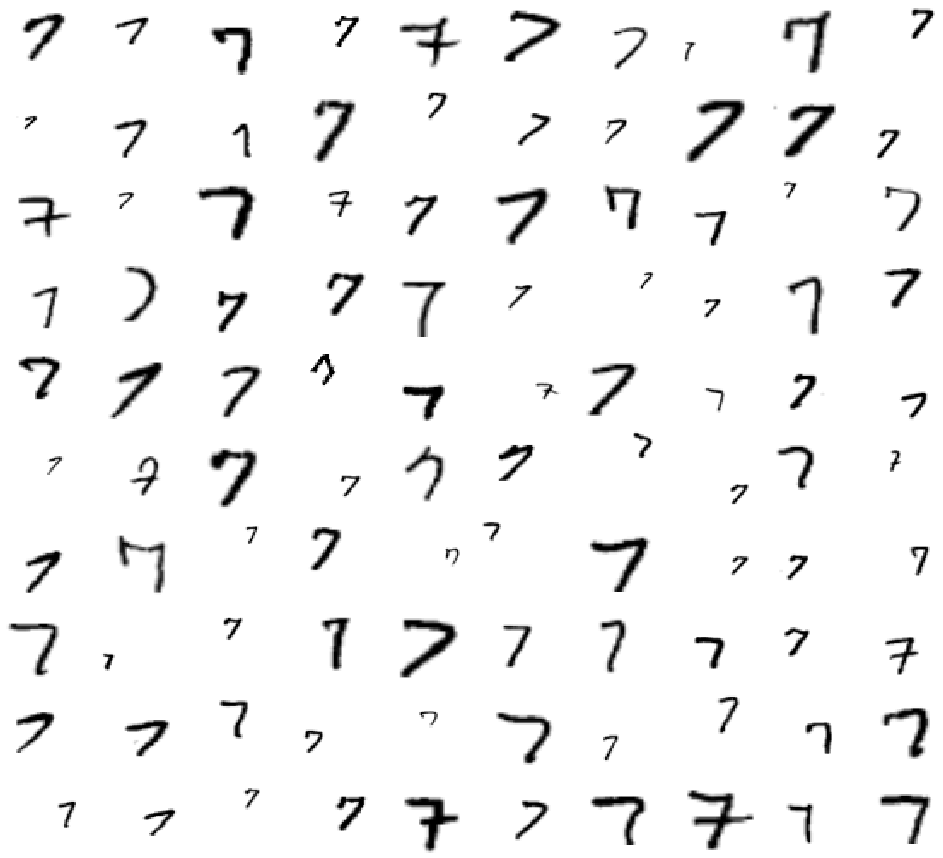}
			\\
			\fbox{\includegraphics[trim={0.3cm 0.3cm 0.3cm 0.3cm},clip,width=0.06\textwidth]{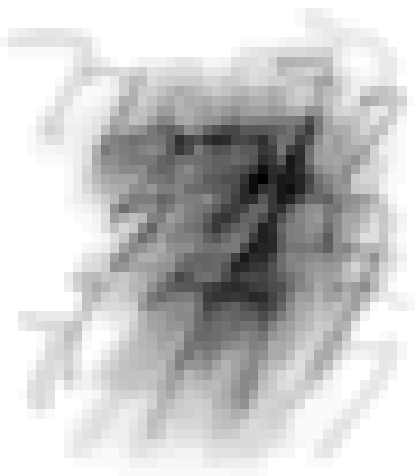}}
			\fbox{\includegraphics[trim={0.8cm 0.5cm 0.5cm 0.5cm},clip,width=0.0585\textwidth]{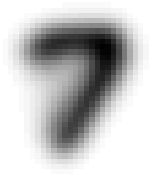}}
			\put(5,20){Wasserstein}
			\put(18,10){Mean}
			\put(-135,20){Euclidean}
			\put(-126,10){Mean}
			\caption{Samples of the digit $7$ from the MNIST dataset and comparison of their Euclidean and Wasserstein Barycenters.}
		\label{fig:sevens}
        \vspace{-0.6cm}
	\end{figure}
	
	For discrete and finite distributions, the WB can be efficiently computed by solving a large linear program~\cite{Anderes2016} or using regularization to approximate a solution efficiently and exploit its convenient algebraic properties~\cite{Agueh2011,Cuturi2014,Cuturi2016}. Here, we consider the problem of computation of WB over a network. The flexibilities induced by the distributed setup make it suitable for problems involving large quantities of data with no centralized storage~\cite{boy11,ned16w,ned17e,Nedic2017a}. Mainly, we assume a group of agents is connected over a network, and each agent locally holds a probability distribution with finite support. The group seeks to compute the WB of all distributions in the network cooperatively. Figure~\ref{fig:erdos_sevens} shows an Erd\H{o}s-R\'enyi random graph with $160$ agents where each agent holds a sample of the digit $7$ from the MNIST dataset.
	
	\vspace{-0.4cm}
	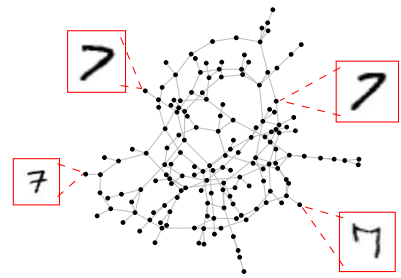
\begin{figure}[ht!]
		\centering
		\begin{tikzpicture}
			[lineDecorate/.style={-,color={rgb:black,1;white,3}},%
			nodeDecorate/.style={shape=circle,inner sep=0.5pt,draw,fill=black},
			scale=0.4]
			\foreach \nodename/\x/\y in {
				1/2.31133198383253/-1.46462564759271,
				2/0.451323427544580/-0.615515896801182,
				3/0.0746895218230484/2.74368460099430,
				4/-3.62848741682743/-1.54430062937836,
				5/2.06788381720618/0.628927681037680,
				6/-1.21768024986900/1.69662992188793,
				7/0.134476550279833/-2.54377970629657,
				8/0.547304534313518/-2.82197690272210,
				9/0.642486796539330/3.79037297277031,
				10/4.69871439522411/-0.433967056801204,
				11/-0.856488433542214/3.24711472249647,
				12/2.55309675735800/1.14108542956293,
				13/4.24703446345803/-0.293411090513647,
				14/-1.56212944903663/-1.99485564237850,
				15/-1.52590343104238/-0.415044980942124,
				16/-2.22881689680023/-2.26965495080501,
				17/0.160862861051241/2.97908973149250,
				18/-0.522983969162077/-1.28989748521989,
				19/2.90318008255349/-0.149972926720722,
				20/1.28628617960708/-0.199957834007207,
				21/-0.354594040650827/2.94645632533678,
				22/0.809083729974355/-2.11795793113196,
				23/0.112511232504780/1.19187603847357,
				24/-0.350127385716271/1.74270316602521,
				25/1.80891051713915/0.660542678445591,
				26/0.914800429411922/3.21927034093369,
				27/-0.187958502490290/-1.37612532984229,
				28/1.73422384356358/-0.408935499423358,
				29/1.97110982875788/1.68005823566926,
				30/0.0421185406781108/0.680834854441936,
				31/-1.73531416319132/-0.429305717246832,
				32/4.71968261146830/-0.248305242837951,
				33/1.49588023604742/4.11230415582639,
				34/1.77341432839360/0.320308383298775,
				35/-0.844599759078817/1.97353356501457,
				36/2.46608706720094/3.23919874199739,
				37/-2.72286409117608/-2.02484488501393,
				38/0.755095119970475/-3.51062534357032,
				39/-3.12079056440620/-2.35385598142887,
				40/2.75047219591697/-1.76391330451792,
				41/-1.69171064655308/2.96707661977244,
				42/0.433705402231638/-1.48037135059837,
				43/0.446514834822805/0.0829235745051586,
				44/-0.599606030136086/-3.04085932523359,
				45/2.38382395912150/-1.05143094925770,
				46/-1.08290741211641/0.386429900198830,
				47/1.85785582304165/4.72289752118605,
				48/-4.36183031555337/-0.747045014812439,
				49/-0.309140215506251/-1.71126244865354,
				50/0.537230112665093/-1.73882796533511,
				51/-1.07304926306476/1.34169051337509,
				52/2.40771346280190/-0.114761783843960,
				53/-1.73834253522964/-1.78757627930449,
				54/1.53582472829470/-0.574384413089401,
				55/-0.549922544697098/2.63927267364263,
				56/-2.05838829466722/-3.02551906230334,
				57/0.854800415916429/2.56660416773585,
				58/-0.335852195726657/-0.951045826769841,
				59/1.43624564075637/3.65246221726961,
				60/0.705326121911779/-3.28327026272103,
				61/-1.71970347091663/0.765394514653884,
				62/1.13994939283807/0.477820087140481,
				63/-1.68248656825299/-0.768084628838691,
				64/-3.88528455904326/-1.36697031086018,
				65/-2.83302332616321/-1.67267896350514,
				66/-0.242903214462637/-0.463090479859048,
				67/-2.38143108701332/2.02341603329265,
				68/1.70032233987568/4.02204565124943,
				69/0.547438789385905/0.383400446672080,
				70/1.68580245902691/4.28773555900818,
				71/-3.87188633479924/-0.761729557169173,
				72/-1.94264891385218/1.72941869115030,
				73/-1.93573231390919/1.47704049643983,
				74/-1.03976402301755/-1.20164778377725,
				75/-0.485447514639244/1.96277089812535,
				76/-0.439817714433498/-2.55544782881305,
				77/0.190730286264780/-0.522340980924927,
				78/-1.42849011893164/1.71618141951939,
				79/1.68742219251062/2.63745106790359,
				80/-1.37426162398665/2.55721140155412,
				81/-1.57245637681011/-0.897310842270736,
				82/3.89695675165026/-0.248369549169562,
				83/-0.262489887278741/2.46126649372949,
				84/0.972127298304970/-2.26446629429443,
				85/1.99353178164566/2.78238157258855,
				86/0.649448808444266/-2.22373575182503,
				87/2.06617226422377/0.889742354454970,
				88/-1.51303124497049/-1.05847143851614,
				89/0.146769177288486/2.48562263400987,
				90/-3.53013807890508/-1.89746043517070,
				91/-0.751149824202475/-0.773891346928523,
				92/-1.82731860290192/-2.54241783144758,
				93/-1.26587207210971/0.553466360773759,
				94/-2.81203254156726/0.0827750061175550,
				95/0.280379198093749/-0.126388316442683,
				96/3.44179847970859/-0.212728985914711,
				97/1.63720727787554/-1.69560176764328,
				98/-0.129237039606143/-0.624531273648156,
				99/1.95037918810916/-0.385665065108770,
				100/1.75097061725556/1.40809186201381,
				101/-2.56548899454670/-2.64121911502280,
				102/-0.983620982700257/1.17852329513141,
				103/1.54994777640271/-0.973681300607485,
				104/2.32045161418362/0.818793708356591,
				105/-0.717403735736036/3.10146933180962,
				106/0.538453304185045/1.05236639513708,
				107/1.07056092999725/-1.13775496598041,
				108/1.59608101653632/0.491780756941001,
				109/1.41299755062284/2.21482162231438,
				110/-1.21319726961716/-0.0278123166417819,
				111/1.95953181211631/0.742813335343626,
				112/-0.146853540811108/3.66524176152126,
				113/1.20258913558083/-0.385598488761495,
				114/2.01884565060200/-0.554802810520709,
				115/-3.27020411764050/-0.308753832381973,
				116/-2.35087767867787/-0.0494700758640008,
				117/2.15007919972341/-1.70002082615080,
				118/0.897718808727636/-3.98532984557968,
				119/-0.317313646326113/-2.14152163758333,
				120/0.708242311283620/-0.0262544659603598,
				121/-1.32556289678617/1.31501579513088,
				122/0.818250526151831/2.74078113611380,
				123/-0.687416126410096/0.808443168018446,
				124/-1.67258516862329/-2.45093976926208,
				125/-1.27158256609629/1.46090584822714,
				126/0.759910569884217/-0.844315343579551,
				127/-2.10620925556653/-2.02575266380638,
				128/-0.747004357313296/-2.53914059840874,
				129/-1.82074227110332/2.13856932861792,
				130/-3.98170730340998/-2.13548462416644,
				131/-0.429651861029726/-3.07362593679426,
				132/1.05411722044178/2.98063139678602,
				133/-1.31929537367884/-1.52257953671742,
				134/2.30396129203470/-0.925710409390245,
				135/1.30083011438487/0.798239932187948,
				136/-0.745078669280549/-0.494299997116139,
				137/-1.15323344048913/-0.803882669220655,
				138/-0.0815965953303751/-2.74139867160865,
				139/-3.73818008212173/-0.194223916413393,
				140/0.601533780806690/-1.19812371775293,
				141/1.77593357999670/-1.31359027625481,
				142/2.60771016842952/0.696657400037764,
				143/0.197956636354943/-1.26662856999655,
				144/-2.53109716132149/-1.25980292482921,
				145/0.236405075087477/-1.86799318130503,
				146/-0.463370229973683/-0.652854073483294,
				147/-1.65883653323396/0.125201523658209,
				148/0.196330048365698/1.57009574258190,
				149/1.95905699382848/0.0863835227716690,
				150/1.52627845390408/1.01154515546266,
				151/-1.18679954578407/-1.29219301000165,
				152/-1.87558590710547/-2.36173843828338,
				153/1.31738526359039/-1.90480149591912,
				154/1.91237573313892/1.31608544114585,
				155/-3.17304336289121/-1.41382770991323,
				156/0.223593899201228/-1.03982218249704,
				157/-1.53011875098192/1.54365972633599,
				158/2.80256903233265/3.56264419010937,
				159/-0.101241220531483/-2.21501432303932,
				160/-2.03721245064823/-0.967072977471189}
			{
				\node (\nodename) at (\x,\y) [nodeDecorate] {};
			}
			\begin{pgfonlayer}{bg} 
			\path
			\foreach \startnode/\endnode in {
				1/1,
				1/40,
				1/45,
				1/97,
				1/134,
				2/2,
				2/77,
				3/3,
				3/55,
				3/83,
				3/122,
				4/4,
				4/155,
				5/5,
				5/25,
				5/28,
				5/29,
				5/142,
				6/6,
				6/24,
				6/61,
				6/80,
				7/7,
				7/84,
				7/128,
				8/8,
				8/38,
				8/60,
				8/145,
				9/9,
				9/59,
				9/112,
				10/10,
				10/13,
				11/11,
				11/80,
				11/112,
				12/12,
				12/87,
				13/13,
				13/32,
				13/96,
				14/14,
				14/124,
				14/151,
				14/152,
				15/15,
				15/81,
				15/110,
				16/16,
				16/53,
				16/101,
				17/17,
				17/89,
				18/18,
				18/58,
				19/19,
				19/52,
				20/20,
				20/34,
				20/99,
				20/103,
				20/150,
				20/156,
				21/21,
				21/83,
				22/22,
				22/50,
				23/23,
				23/30,
				23/148,
				24/24,
				24/35,
				24/75,
				24/83,
				24/106,
				25/25,
				25/62,
				25/104,
				26/26,
				26/122,
				27/27,
				27/66,
				27/159,
				28/28,
				28/52,
				28/126,
				28/141,
				29/29,
				29/79,
				30/30,
				30/95,
				30/106,
				30/123,
				31/31,
				31/46,
				31/137,
				31/144,
				31/147,
				32/32,
				33/33,
				33/59,
				34/34,
				34/87,
				35/35,
				35/55,
				35/121,
				36/36,
				36/85,
				36/158,
				37/37,
				37/39,
				37/65,
				37/90,
				37/92,
				37/144,
				38/38,
				38/118,
				39/39,
				40/40,
				41/41,
				41/80,
				42/42,
				42/50,
				42/156,
				43/43,
				43/62,
				43/66,
				43/69,
				44/44,
				44/76,
				45/45,
				45/114,
				46/46,
				46/51,
				46/66,
				46/121,
				47/47,
				47/70,
				48/48,
				48/71,
				49/49,
				49/76,
				49/119,
				49/145,
				49/146,
				50/50,
				50/143,
				51/51,
				51/75,
				51/78,
				51/157,
				52/52,
				52/96,
				52/108,
				52/114,
				53/53,
				53/74,
				53/127,
				54/54,
				54/113,
				55/55,
				55/105,
				56/56,
				56/92,
				57/57,
				57/89,
				57/109,
				57/132,
				58/58,
				58/77,
				58/126,
				58/137,
				58/151,
				59/59,
				59/68,
				59/70,
				59/79,
				60/60,
				61/61,
				61/73,
				61/110,
				61/116,
				62/62,
				62/120,
				62/135,
				62/150,
				63/63,
				63/74,
				63/116,
				63/151,
				64/64,
				64/71,
				64/90,
				65/65,
				65/144,
				66/66,
				66/143,
				67/67,
				67/72,
				68/68,
				69/69,
				70/70,
				71/71,
				71/115,
				72/72,
				72/121,
				73/73,
				73/129,
				74/74,
				74/81,
				74/133,
				74/156,
				75/75,
				75/89,
				76/76,
				76/131,
				77/77,
				77/98,
				77/120,
				78/78,
				79/79,
				79/109,
				80/80,
				80/129,
				81/81,
				81/160,
				82/82,
				82/96,
				83/83,
				84/84,
				84/97,
				85/85,
				85/109,
				86/86,
				86/153,
				86/159,
				87/87,
				87/150,
				88/88,
				88/137,
				89/89,
				90/90,
				90/130,
				91/91,
				91/146,
				92/92,
				92/128,
				93/93,
				93/123,
				93/147,
				94/94,
				94/116,
				95/95,
				95/113,
				95/146,
				96/96,
				97/97,
				97/103,
				98/98,
				98/146,
				99/99,
				99/134,
				99/149,
				100/100,
				100/150,
				101/101,
				102/102,
				102/123,
				102/125,
				103/103,
				104/104,
				105/105,
				106/106,
				106/108,
				107/107,
				107/126,
				108/108,
				108/111,
				108/149,
				109/109,
				109/122,
				109/150,
				110/110,
				110/146,
				111/111,
				112/112,
				113/113,
				113/114,
				114/114,
				115/115,
				115/116,
				115/139,
				116/116,
				117/117,
				117/141,
				118/118,
				119/119,
				119/159,
				120/120,
				121/121,
				121/125,
				122/122,
				123/123,
				124/124,
				125/125,
				126/126,
				126/140,
				127/127,
				128/128,
				128/159,
				129/129,
				130/130,
				131/131,
				132/132,
				133/133,
				134/134,
				135/135,
				136/136,
				136/146,
				137/137,
				138/138,
				138/159,
				139/139,
				140/140,
				140/143,
				141/141,
				141/153,
				142/142,
				143/143,
				143/145,
				144/144,
				144/155,
				145/145,
				145/156,
				146/146,
				147/147,
				148/148,
				149/149,
				150/150,
				150/154,
				151/151,
				152/152,
				153/153,
				154/154,
				155/155,
				156/156,
				157/157,
				158/158,
				159/159,
				160/160}
			{
				(\startnode) edge[lineDecorate] node {} (\endnode)
			};
		\node at (-4,3) {\cfbox{red}{\includegraphics[scale=0.8]{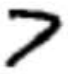}}};
		\node at (-6,-1) {\cfbox{red}{\includegraphics[scale=0.8]{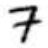}}};
		\node at (5,2) {\cfbox{red}{\includegraphics[scale=0.8]{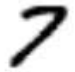}}};
		\node at (5,-3) {\cfbox{red}{\includegraphics[scale=0.8]{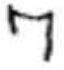}}};
		\draw[red,dashed] (-5.3,-1.5) -- (-4.4,-0.7);
		\draw[red,dashed] (-5.3,-0.4) -- (-4.4,-0.7);
		\draw[red,dashed] (4.2,-2.2) -- (2.8,-1.8);
		\draw[red,dashed] (4.2,-3.8) -- (2.8,-1.8);
		\draw[red,dashed] (-3.2,3.8) -- (-2.5,2.1);
		\draw[red,dashed] (-3.2,2.2) -- (-2.5,2.1);
		\draw[red,dashed] (4.1,1.2) -- (2.1,1.7);
		\draw[red,dashed] (4.1,2.8) -- (2.1,1.7);
			\end{pgfonlayer}
			\end{tikzpicture}
		\caption{Erd\H{o}s-R\'enyi random graph where each agent privately holds a sample of the digit $7$ from the MNIST dataset.} 
		\label{fig:erdos_sevens}
	\end{figure}
	
	\vspace{-0.4cm}
	Distributed consensus with the Wasserstein metric was introduced in~\cite{Bishop2014a,Bishop2014}. In~\cite{Bishop2014}, the authors showed asymptotic converge to the WB of the initial distributions given some weak connectivity assumptions on the graph over which agents exchange information. Nevertheless, the proposed algorithm requires that each agent computes an exact WB of local distributions at each iteration. Although one can have closed-form solutions for some families of continuous distributions, in general, the problem can be intractable. On the other hand, a recent approach~\cite{Staib2017} explores the computational advantages of a dual formulation of the WB and exploits the parallelizable structure of the problem to propose a scalable, and communication-efficient algorithm for its computation on arbitrary continuous distributions. Nevertheless, it requires a central fusion center that coordinates the actions of the parallel machines.
	
	In contrast with existing literature~\cite{Bishop2014,Staib2017}, \dd{we} propose a first-order algorithm that can be executed distributedly over a network with unknown topology. We derive an explicit convergence rate of the order $O(1/k^2)$ with an additional cost that depends on the condition number of the graph over which the agents interact. Additionally, we present two numerical examples to illustrate and validate our results. First, we show some basic properties of the algorithm for the problem of computing WB of univariate, discrete and truncated Gaussian distributions. Then, we show the result of applying our algorithm to a subset of the MNIST digit database on a large-scale network of $1000$ agents.
	
	This paper is organized as follows. Section~\ref{sec:problem} presents basic definitions for the problem of computation of WB over networks. Section~\ref{sec:results} states auxiliary results, introduces the proposed algorithm and states our main results on its convergence rate and dependency on the problem parameters. Section~\ref{sec:numerical} shows two numerical examples that experimentally verify the theoretical properties of the algorithm. Conclusions and future work are presented in Section~\ref{sec:conclusions}. 
	
	\textbf{Notation:} 
	We assume that the agents are indexed from $1$ through $m$. The enumeration is not needed in the execution of the proposed algorithm. It is only used in the analysis. Superscripts $i$ or $j$ denote agent indices and subscript $k$ \dd{defines} iteration indices. $[A]_{ij}$ denotes the $i$-th row and $j$-th column entry of a matrix $A$. $I_{n}$ is the identity matrix of size $n$. \cu{For a linear operator $A:E\to H$, where $E$ and $H$ are two finite-dimensional real vector spaces with duals $E^*$ and $H^*$, and norms $\|\cdot\|_E$ and $\|\cdot\|_H$ respectively, we define its norm as $\|A\|_{E\to H} =\max_{x \in E,u \in H^*} \{ \left< u,Ax\right> \mid \|x\|_E =1, \|u\|_{H^*} = 1 \}$}. 
	\cu{For matrices $A$ and $B$, $A \circ  \dd{B}$ and $A/B$ stands for the element-wise product and division, respectively. We use {\bf{1}} to denote a column vector with all entries equal to $1$. }
	
	\section{Problem Statement}\label{sec:problem}
	
	In this section, we recall some basic definitions of the optimal transport problem. We describe the Wasserstein distance and the WB of a set of discrete probability distributions with finite support. Finally, we introduce the problem of distributed computation of WB over networks.
	
	\subsection{Entropy Regularized Optimal Transport}
	
	Consider two probability distributions $p,q \in S_1(n)$ with support on a finite set of points $\{x_i \in \mathbb{R}^d \}_{i=1}^n$ such that $p(x_i) = p_i$ and $q(x_i) = q_i$, where $S_1(n)  = \{ p \in \mathbb{R}_+^n  \mid p^T \boldsymbol1 =1 \}$. Moreover, consider a non-negative symmetric matrix \mbox{$M$}, where $[M]_{ij}\in \mathbb{R}_+$ accounts for the cost of moving mass from $p_i$ to bin $q_j$. Without loss of generality, in the numerical example we will consider the Euclidean costs where $[M]_{ij} = \|x_i - x_j\|_2^2$. Additionally, define the set of couplings or \textit{transportation polytope} $U(p,q)$  as
	
	\vspace{-0.3cm}
	{\small
		\begin{align*}
		U(p,q) & \triangleq \left\lbrace X \in \mathbb{R}_+^{n \times n} \mid X \boldsymbol{1} = p, X^T\boldsymbol{1} = q \right\rbrace. 
		\end{align*}
	}
	
	\vspace{-0.5cm}
	The entropy-regularized OT problem~\cite{Cuturi2013} seeks to minimize the transportation costs while maximizing the entropy (maximum-entropy principle) and is defined as
	
	\vspace{-0.4cm}
	{\small\begin{align}\label{eq:wass_distance}
		\mathcal{W}_\gamma (p,q) &\triangleq \min_{X \in U(p,q)} \left\lbrace \left\langle  M,X\right\rangle - \gamma E(X)\right\rbrace,
		\end{align}}
	\vspace{-0.4cm}
	
	\noindent where $\gamma>0$, and
	
	\vspace{-0.4cm}
	{\small\begin{align}
		\left\langle M,X \right\rangle & \triangleq  \sum_{i=1}^{n} \sum_{j=1}^{n} M_{ij} X_{ij} \ \ 
		\text{and} \ \
		E(X) \triangleq - \sum_{i=1}^{n} \sum_{j=1}^{n} h(X_{ij}), \nonumber
		\end{align}}
	
	\vspace{-0.3cm}
	\noindent and $\forall x>0, h(x) \triangleq x\log x$ and $h(0) \triangleq 0$. A solution $\mathcal{W}_0(p,q)$ is called the Wasserstein distance between $p$ and $q$ \dd{and if $\gamma >0$, $\mathcal{W}_\gamma(p,q)$ is known as regularized (or smoothed) Wasserstein distance. For $\gamma >0$}, problem~\dd{\eqref{eq:wass_distance}} is \dd{strongly} convex and admits a unique solution \dd{$X^*$}. 
	
	\dd{For simplicity, let us introduce the notation $\mathcal{W}_{\gamma,q}(p)$ for fixed probability distribution $q \in S_1(n)$}
	
	\vspace{-0.6cm}
	\begin{align*}
	\mathcal{W}_{\gamma,q}(p) \triangleq  \mathcal{W}_{\gamma}(p,q).
	\end{align*}
	
	\vspace{-0.2cm}
	One particular advantage of entropy-regularizing the Wasserstein distance is that there exists closed-form representations for the dual problem and its gradients~\cite{Agueh2011,Cuturi2016} where the Fenchel-Legendre transform of~\eqref{eq:wass_distance} is defined as
	
	\vspace{-0.4cm}
	{\small
		\begin{align}\label{eq:dual_wass}
		\mathcal{W}^*_{\gamma,q}(y) & \triangleq \max_{p \in S_1(n)}\left\lbrace  \left\langle y,p\right\rangle  - \dd{\mathcal{W}_{\gamma,q} (p)}\right\rbrace .
		\end{align}
	}
	
	\vspace{-0.4cm}
	\cu{In \cite{Blondel2017}, other regularization functions were explored. The squared $2$-norm was shown to produce sparse transportation plans. In this paper, we will use the entropy regularization. Nevertheless, our results extend naturally to other regularization functions, especially those that admit closed-form solution of dual gradients. The next theorem states the closed-form solutions of the dual problem, and the gradient of the entropy regularized WB problem.}
	
	\vspace{-0.2cm}
	\begin{theorem}[Theorem $2.4$ in~\cite{Cuturi2016}]\label{thm:cuturi}
		For $\gamma >0$, the Fenchel-Legendre dual function $\mathcal{W}^*_{\gamma,q}(y)$ is differentiable and its gradient $\nabla \mathcal{W}^*_{\gamma,q}(y)$ is $1/\gamma$-Lipschitz \dd{in the 2-norm} with
		
		\vspace{-0.4cm}
		{\small
			\begin{align*}
			\mathcal{W}^*_{\gamma,q}(y) & = \gamma\left(E(q) + \left\langle q, \log K \alpha \right\rangle  \right) \ \ \text{and}\\
			\nabla \mathcal{W}^*_{\gamma,q}(y) & = \alpha \circ \left(K \cdot {q}/({K \alpha}) \right) \in S_1(n),
			\end{align*}
		}
		
		\vspace{-0.5cm}
		\noindent    where $y \in \mathbb{R}^n$, $\alpha = \exp( {y}/{\gamma}) $ and \mbox{$K = \exp( {-M}/{\gamma }) $}.
	\end{theorem}
	
	\vspace{-0.1cm}
	We will use the result\dd{s of}  Theorem~\ref{thm:cuturi} to design an algorithm for the computation of the WB on graphs based \dd{on} recent ideas of dual approaches for convex optimization problems with affine constrains~\cite{all17,Gasnikov2016} and optimal algorithms for distributed optimization~\cite{Uribe2017,uribe2018dual}.

	\subsection{Computation of a Wasserstein Barycenter over a Network}
	
	The uniform WB~\cite{Agueh2011,Cuturi2014} of a family of discrete distributions $q_i\in S_1(n)$, for $i=1,\ldots,m$; is defined as the solution to the following optimization problem
	
	\vspace{-0.4cm}
	{\small
		\dd{\begin{align}\label{w_barycenter}
			\min_{p \in S_1(n)} \sum\limits_{i=1}^{m} \mathcal{W}_{\gamma,q_i}(p).
			\end{align}}}
	
	\vspace{-0.3cm}
	The WB is an extension of the Euclidean barycenter to nonlinear metric spaces \dd{corresponding} to the empirical Fr\'{e}chet mean~\cite{Frechet1948}. The existence and uniqueness of WB has been studied in the literature~\cite{Bigot2016}. Problem~\eqref{w_barycenter} is strictly convex and admits a unique solution, denoted by $p^*$~\cite{Cuturi2016}.
	
	\dd{For the distributed computation of WBs, let us introduce the stacked column vectors} ${{{\mathtt{p}} = [p_1^T,\cdots,p_m^T]^T}}$ and ${{{\mathtt{q}} = [q_1^T,\cdots,q_m^T]^T}}$, where 
	\dd{$p_i,q_i \in S_1(n)$}, and rewrite the problem \eqref{w_barycenter} in an equivalent form
	
	\vspace{-0.4cm}
	{\small
		\dd{\begin{align}\label{w_barycenter2}
			\min_{\substack{p_1=\cdots=p_m \\ p_1,\dots,p_m \in S_1(n)}}  ~ \sum\limits_{i=1}^{m} \mathcal{W}_{\gamma,q_i}(p_i).
			\end{align}}
	}
	
	
	\vspace{-0.3cm}
	We denote the unique solution of~\eqref{w_barycenter2} by \mbox{$\mathtt{p}^* = [(p^*)^T,\cdots,(p^*)^T]^T$}.
	
	We seek to solve problem~\eqref{w_barycenter2} in a distributed manner over a network, where each distribution $q_i$ is held by an agent $i$ on a network. We model such a network as a fixed \textit{connected undirected graph} \mbox{$\mathcal{G} = (V,E)$}, where 
	$V$ is the set of $m$ nodes, and $E$ is a set of edges. We assume that the graph $\mathcal{G}$ does not have self-loops. The network structure imposes information constraints; specifically, each node $i$ has access to $q_i$ only and a node can exchange information only with its immediate neighbors, i.e., a node $i$ can communicate with node $j$ if and only if $(i,j)\in E$. 
	
	We can represent the communication constraints imposed by the network by introducing 
	a set equivalent to the constraints in~\eqref{w_barycenter2}.
	To do so, we define the Laplacian matrix 
	$\bar W{\in \mathbb{R}^{m\times m}}$ of the graph $\mathcal{G}$ by
	{\small\begin{align*}
		[\bar W]_{ij} = \begin{cases}
		-1,  & \text{if } (i,j) \in E,\\
		\text{deg}(i), &\text{if } i= j, \\
		0,  & \text{otherwise,}
		\end{cases}
		\end{align*}}
	
	\vspace{-0.3cm}
	\noindent where $\text{deg}(i)$ is the degree of the node $i$, i.e., the number of neighbors of the node. \cu{Moreover, we denote as $d_{\max}$ and $d_{min}$ the maximum and \dd{the} minimum degree among all nodes in the network.}
	Finally, define the communication matrix (also referred to as an interaction matrix) 
	by \mbox{$W \triangleq\bar W \otimes I_n$}, where $\otimes$ indicates the Kronecker product.
	
	Throughout the paper, {\it we assume that graph $\mathcal{G} = (V,E)$ is undirected and connected}.
	Under this assumption, the Laplacian matrix $\bar W$ is symmetric and positive \dd{semidefinite}. Furthermore,
	the vector $\boldsymbol{1}$ is the unique (up to a scaling factor) eigenvector associated with the eigenvalue
	$\lambda=0$. $W$ inherits the properties of $\bar W$, including symmetry and positive semidefiniteness. Moreover, $W{x} = 0$ if and only if {$x_1 = \cdots = x_m$}, and $\sqrt{W}{x} = 0$ if and only if {$x_1 = \cdots = x_m$}.
	Therefore, one can equivalently rewrite problem~\eqref{w_barycenter2} as
	
	\vspace{-0.3cm}
	{\small
		\dd{    \begin{align}\label{consensus_problem2}
			\min_{\substack{p_1,\dots, p_m \in S_1(n) \\ \sqrt{W} \mathtt{p}=0 }} ~ \mathcal{W}_{\gamma,\mathtt{q}}(\mathtt{p}) = \sum\limits_{i=1}^{m} \mathcal{W}_{\gamma, q_i}(p_i) .
			\end{align}}
	}
	
	Note that the constraint set \dd{$ \{p_1,\dots, p_m \in S_1(n) \mid \sqrt{W} \mathtt{p}=0\}$}
	is 
	the same as the set \dd{$\{p_1,\dots, p_m \in S_1(n) \mid p_1 = \cdots  = p_m\}$},
	since
	\mbox{$\ker (\sqrt{W}) = \spn(\boldsymbol{1})$} due to the connectivity of the graph $\mathcal{G}$. 
	
	In the next section, we state the proposed algorithm for solving~\eqref{consensus_problem2} and analyze its convergence rate.
	
	\section{Algorithm and Results}\label{sec:results}
	
	In this section, we build on recent results on dual approaches for optimal distributed optimization~\cite{Uribe2017,ani17} to construct an algorithm that solves~\eqref{consensus_problem2} over a network. Moreover, we analyze its convergence rates and provide explicit dependencies on the problem parameters and the network topology.

	\subsection{Dual Approach for Strongly Convex Functions and Affine Constraints}
	\label{S:DualApprMain}
	In~\cite{ani17}, the authors proposed a novel analysis for the minimization of strongly convex functions with affine constraints of the form
	
	\vspace{-0.9cm}
	\begin{align}\label{eq:linear}
	\min_{Ax=0}f(x),
	\end{align}
	where $f(x)$ is $1$-strongly convex with respect to the \mbox{$p$-norm} with the corresponding dual problem defined as
	
	\vspace{-0.5cm}
	\begin{align}
	\label{eq:dual_to_linear}
	\min_y g(y)  \ \ \ \text{where}  \ \ g(y) = \max_x \{\left\langle A^Ty,x \right\rangle -f(x)\}.
	\end{align}
	
	\vspace{-0.2cm}
	\pd{We denote $x^*(A^Ty)$ the solution to the problem defining $g(y)$.}
	The dual function $g(y)$ is $L$-smooth with \mbox{\cu{$L = \|A\|_{L^1 \to L^2}= \max_{i=1,\cdots,m}\|A_i\|_2^2$}}, where $A_i$ is the $i$-th column of $A$. Thus, one can use accelerated first order methods such as Nesterov's Fast Gradient~\cite{nes83} or one of its recent reformulations~\cite{all14} to obtain an approximate solution. The novelty in~\cite{ani17} lies in the statement of the convergence rate of the accelerated methods in terms of the duality gap and the constraint violation. Additionally, it was shown that for the linear coupling accelerated algorithm~\cite{all17} one can guarantee that the solutions will remain in a closed ball around the optimal solution, with a radius proportional to the distance between the initial point of the algorithm and the optimal solution. Next, we state a technical result, based on~\cite{ani17}, that will help us in the design and analysis of our proposed algorithm for the distributed computation of the WB.
	
	\vspace{-0.3cm}
	\begin{theorem}
		\label{thm:gasnikov}
		The fast gradient method based on linear coupling proposed in~\cite{all14} with the change $y_k$ to $w_k$ and $x_k$ to $y_k$ and applied to \pd{problem~\eqref{eq:dual_to_linear}, with $w_0 = y_0 = z_0 = 0$}, has the following properties: 
		$ \dd{\forall ~} k \geq N$ and $\varepsilon >0$, it holds that
		\begin{align*}
		g(\pd{w_k}) + f(\breve x_k) \leq \varepsilon \qquad \text{and} \qquad \|A\breve x_k\|_{\pd{2}}\leq  \varepsilon / R,
		\end{align*}
		where \cu{ $
			\breve{x}_k = \sum_{t=0}^{k-1} \frac{(t+2)}{k(k+3)} x^*(A^Ty_{t+1})
			$}, $N \triangleq \sqrt{\pd{16} \cu{L} R^2/\varepsilon}$, \cu{$R = \|y^*\|_2 < \infty$} and $y^*$ is the optimal point of $g(\cdot)$ with minimal norm.
		
	\end{theorem}
	\pd{
		\begin{proof}
			The proof consists in combining Theorem 2 in \cite{dvu16} and proof of Theorem 1 in \cite{dvu16b}.
		\end{proof}
	}
	\pd{\subsection{Dual Approach for Wasserstein Barycenter Problem}}
	\label{S:DualWasser}
	\dd{Problem~\eqref{consensus_problem2} can be equivalently reformulated as the maximization problem
		\vspace{-0.2cm}
		{\small
			\begin{align*}
			\max_{\substack{p_1,\dots, p_m \in S_1(n) \\ \sqrt{W} \mathtt{p}=0 }}-\mathcal{W}_{\gamma,\mathtt{q}}(\mathtt{p}) = \sum\limits_{i=1}^{m} \mathcal{W}_{\gamma, q_i}(p_i)
			\end{align*}}
		\vspace{-0.2cm}
		\noindent  with its corresponding Lagrangian dual problem
		{\small
			\begin{align*}
			\min_\cu{\mathtt{y}} ~\max_{p_1,\dots, p_m \in S_1(n) } ~ \left\lbrace\langle \mathtt{y}, \sqrt{W}\mathtt{p}\rangle-\mathcal{W}_{\gamma,\mathtt{q}}(\mathtt{p})\right\rbrace,
			\end{align*}}
		where $\mathtt{y} = [y_1^T \cdots y_m^T]^T$. \\
		~\\
		Moreover, the Fenchel-Legendre transform of $\mathcal{W}_{\gamma,q_i}(p_i)$ is 
		{\small
			\begin{align*}
			\mathcal{W}^*_{\gamma,q_i}([\sqrt{W}\mathtt{y}]_i) & = \max_{p_i \in S_1(n)}  \left\lbrace \left\langle [\sqrt{W}\mathtt{y}]_i,p_i \right\rangle -\mathcal{W}_{\gamma,q_i}(p_i) \right\rbrace,
			\end{align*}}
		\noindent where $[\sqrt{W}y]_i=\sum_j^m \sqrt{W}_{ij}y_j$, and $\sqrt{W}_{ij} = [\sqrt{\bar W}]_{ij} \otimes I_n$. Therefore, we can rewrite the problem~\eqref{consensus_problem2} as follows
		\vspace{-0.2cm}
		{\small
			\begin{align}\label{dual_problem22}
			\min_{\mathtt{y}} \mathcal{W}^*_{\gamma,\mathtt{q}}(\sqrt{W}\mathtt{y}) = \sum_{i=1}^{m} \mathcal{W}^*_{\gamma, q_i}([\sqrt{W}\mathtt{y}]_i).
			\end{align}}}
	
	\vspace{-0.2cm}
	Additionally, from Theorem~\ref{thm:cuturi} the gradient can be expressed in closed form as 
	
	\vspace{-0.4cm}
	{\small\begin{align*}
		\nabla \mathcal{W}^*_{\gamma,q_i}\left([\sqrt{W}y]_i\right) & = \sum_{j=1}^{m}\sqrt{W}_{ij} p^*_j\left([\sqrt{W}y]_j\right), 
		\end{align*}}
	
	\vspace{-0.3cm}
	\noindent where $p^*_j(\tilde y_j)  = \alpha(\tilde y_j) \circ ( K \cdot { q_i}/{(K \alpha(\tilde y_j) )} )$. Moreover, it holds that one can recover the solution $\mathtt{p}^*$ to the primal problem~\eqref{consensus_problem2} from a solution $\mathtt{y}^*$ to the dual problem~\eqref{dual_problem22} as $
	\mathtt{p}^* = \mathtt{p}^*(\sqrt{W}\mathtt{y}^* )$.
	
	The optimality relation between the dual and the primal problem follows from Theorem $3.1$ in~\cite{Cuturi2016}. In general, the dual problem~\eqref{dual_problem22} can have multiple solutions of the form $\mathtt{y}^* + \ker(\sqrt{W})$ when the matrix $\sqrt{W}$ does not have a full row rank. When the solution is not unique, we {\it will use $\mathtt{y}^*$ to denote the smallest norm solution}, and we let $R$ be its norm, i.e. $R = \|\mathtt{y}^*\|_2$.
	
	\cu{The entropy regularization term is $\gamma$-strongly convex with respect to the $1$-norm over the probability simplex $S_1(n)$}. As a consequence, the computation of the WB of a set of discrete probability distributions $\{q_i\}_{i=1}^m$ is equivalent to solving the dual decomposable $L$-smooth (with respect to the $2$-norm) optimization problem~\eqref{dual_problem22} with \mbox{$\cu{L = \|\sqrt{W}\|^2_{L^1 \to L^2}}/\gamma$}~\cite{Kakade2009}. Specifically, in this setup it holds that
	
	\vspace{-0.3cm}
	{\small \begin{align*}
		\|\sqrt{W}\|^2_{L^1 \to L^2} &=\max_{i=1,\cdots,m}\|\sqrt{W}_i\|_2^2 = \dd{\max_{i=1,\cdots,m}\sqrt{W}_i^T\sqrt{W}_i}\\
		&= \max_{i=1,\cdots,m} [W]_{ii}= d_{\max}.
		\end{align*}}
	
	
	
	\vspace{-0.2cm}
	\pd{\subsection{Algorithm and Main Results}}
	\label{S:AlgAndMain}
	We can explicitly write the Nesterov's Accelerated Gradient Method (FGM)~\cite{nes13} for smooth functions. Particularly, we use follow the linear coupling approach recently proposed in~\cite{all14}. Setting $\hat{\mathtt{w}}_{k} = \hat{\mathtt{z}}_{k}=\hat{\mathtt{y}}_{k} = \boldsymbol{0}$, the FGM generates iterates according to:
	
	\vspace{-0.4cm}
	{\small
		\begin{subequations}\label{eq:nesterov}
			\begin{align}
			\hat{\mathtt{y}}_{k+1} & = \tau_k \hat{\mathtt{z}}_{k} + (1 - \tau_k) \hat {\mathtt{w}}_{k}\\
			\hat {\mathtt{w}}_{k+1} & = \hat{\mathtt{y}}_{k+1}  - \cu{\frac{1}{L }} \sqrt{W} \mathtt{p}^*\left(\sqrt{W} \hat{\mathtt{y}}_{k+1} \right) \\
			\hat {\mathtt{z}}_{k+1} & = \hat{\mathtt{z}}_{k}  - \alpha_{k+1} \sqrt{W} \mathtt{p}^*\left(\sqrt{W} \hat{\mathtt{y}}_{k+1}  \right)
			\end{align}
	\end{subequations}}
	where $\alpha_{k+1} = (k+2)/(2L)$ and $\tau_k = 2/(k+2)$.
	
	Unfortunately, algorithm~\eqref{eq:nesterov} cannot be executed in a distributed manner. Although the entries of local gradient vectors can be computed independently by each node, the sparsity pattern of the matrix $\sqrt{W}$ need not be the same as the communication constraints induced by the graph $\mathcal{G}$. Thus, the variables $\hat {\mathtt{w}}_{k}$ and $\hat {\mathtt{z}}_{k}$ cannot be computed on the network. This problem is solved by a change of variables such that $\tilde{ \mathtt{y}} = \sqrt{W}\hat{\mathtt{y}}$, $\tilde{ \mathtt{w}} = \sqrt{W}\hat{\mathtt{w}}$ and $\tilde{ \mathtt{z}} = \sqrt{W}\hat{\mathtt{z}}$.
	
	Algorithm~\ref{alg:main} presents the resulting distributed accelerated gradient method for the dual problem of the WB problem.

	\begin{algorithm}[ht]
		\caption{Distributed Computation of WB}
		\label{alg:main}
		\begin{algorithmic}[1]
			\dd{\REQUIRE Each agent $i\in V$ is assigned its distribution $q_i$.}
			\STATE{All agents set $\tilde w_0^i = \tilde y_0^i = \tilde z_0^i = \boldsymbol{0} \in \mathbb{R}^n$ and $N$}
			\STATE{For each agent $i \in V$:}
			\FOR{ $k=0,1,2,\cdots,N\dd{-1}$ }
			\STATE{$\tau_k = \frac{2}{k+2} \  \text{and} \  \alpha_{k+1} = \frac{k+2}{2}\cu{\frac{1}{L }}  $}
			\STATE{$\tilde y_{k+1}^i = \tau_k\tilde z_k^i + (1-\tau_k) \tilde w_k^i$}
			\STATE{$p^*_i(\tilde{y}^i_{k+1}) \scriptstyle{= \exp(\tilde y_{k+1}^i/ \gamma )\circ\left( K \cdot \frac{q_i}{ K\exp(\tilde y_{k+1}^i/\gamma)  } \right)  }$}
			\STATE{Share $p^*_i(\tilde{y}^i_{k+1})$ with $\{j \mid (i,j) \in E \}$}
			\STATE{$\tilde w_{k+1}^i  = \tilde y_{k+1}^i - \cu{\frac{1}{L }} \sum_{j=1}^{m} W_{ij} p^*_j(\tilde y^j_{k+1})$ }
			\STATE{$\tilde z^i_{k+1} = \tilde z^i_k -\alpha_{k+1} \sum_{j=1}^{m} W_{ij} p^*_j(\tilde y^j_{k+1})$}
			\ENDFOR
			\STATE{\pd{
					Set $( y_N^*)_i = \tilde w^i_N $, $\forall i\in V$        
			}}
			\STATE{\pd{
					Set $(p^*_N)_i = \sum_{k=0}^{N-1} \frac{(k+2)}{N(N+3)}p^*_i(\tilde y^i_{k+1})$, $\forall i\in V$       
			}}
		\end{algorithmic}
	\end{algorithm}
	
	\dd{Based on \cite{ani17}, we can guarantee that Algorithm \eqref{alg:main} generates sequences of vectors $\{\tilde y^i_k,\tilde w^i_k,\tilde z^i_k\}$ which remain in a ball $B_R(0)$ with $R=\|\tilde y^i_0 - \tilde y_i^*\|_2 = \|\tilde y_i^*\|_2$}. Now, we are ready to state our main result that provides a convergence rate for Algorithm~\ref{alg:main} with explicit dependencies on the problem parameters and the topology of the network. 
	
	

	\vspace{-0.3cm}
	\begin{theorem}\label{thm:main}
		Let $\varepsilon>0$ and assume that  $\|\nabla \mathcal{W}^*_{\gamma,\mathtt{q}}(\tilde{\mathtt{y}})\|_2 \leq \dd{G} $ on a ball $B_R(0)$. Then, it holds that that after 
		
		\vspace{-0.2cm}
		{\small 
			\begin{align*}
			N & \geq \sqrt{ \frac{16 G^2}{\gamma \cdot \varepsilon}  \frac{d_{\max}}{d_{\min}} }
			\end{align*}
		}
		
		\vspace{-0.2cm}
		\noindent iterations, the outputs of Algorithm~\ref{alg:main}, i.e. $\mathtt{p}^*_N = [(p_N^*)_1^T,\cdots,(p_N^*)_m^T]^T$ and $\mathtt{y}^*_N = [(y_N^*)_1^T,\cdots,(y_N^*)_m^T]^T$ have the following properties: 
		\begin{align*}
		\mathcal{W}_{\gamma,\mathtt{q}}(\mathtt{p}^*_N)+ \mathcal{W}^*_{\gamma,\mathtt{q}}(\mathtt{y}^*_N)  
		\leq  \varepsilon \ \ \text{and} \ \ \|\sqrt{W}\mathtt{p}^*_N\|_2 \leq {\varepsilon}/{R}.
		\end{align*} 
	\end{theorem}
	
	\begin{proof}
		The dual function $\mathcal{W}^*_{\gamma,\mathtt{q}}(\mathtt{y})$ is $(\cu{d_{\max}}/\gamma)$-smooth. Thus, from Theorem~\ref{thm:gasnikov} it follows that
		\begin{align*}
		\mathcal{W}_{\gamma,\mathtt{q}}(\mathtt{p}^*_N)+ \mathcal{W}^*_{\gamma,\mathtt{q}}(\mathtt{y}^*_N)  
		\leq  \varepsilon \ \ \text{and} \ \ \|\sqrt{W}\mathtt{p}^*_N\|_2 \leq {\varepsilon}/{R}.
		\end{align*} 
		holds for $k \geq \sqrt{16 \cu{d_{\max} }R^2/(\gamma \varepsilon)}$. \\Moreover, 
		\dd{considering the boundedness of the gradients of the dual function and by Theorem~$3$ in~\cite{lan17}, we can estimate the radius as}      
		\dd{{\small \begin{align*}
				R^2 \leq \frac{\|\nabla \mathcal{W}^*_{\gamma,\mathtt{q}}(\mathtt{y}^*)\|_2^2}{ \cu{\min\limits_{\substack{ \scriptscriptstyle{x \in E \perp \ker(\sqrt{W}),u \in H^* }\\ \scriptscriptstyle {\|x\|_E =1, \|u\|_{H^*} = 1} } } \{ \langle u,\sqrt{W}x \rangle  \}}} =\frac{ G^2}{d_{\min}}.
				\end{align*}} }
		Thus, we require $ k \geq {\scriptscriptstyle \sqrt{\frac{16 \cu{G}^2}{\gamma \cdot \varepsilon}\cu{\frac{d_{\max}}{d_{\min}} }}}$ and the desired result follows.
	\end{proof}

	Theorem~\ref{thm:main} provides an estimate of the minimum number of iterations required for the proposed algorithm to reach some arbitrary relative accuracy in the solution of the distributed WB problem. The convergence rate is shown to be of the order $O(1/k^2)$ which has been established to be optimal for smooth convex optimization problems \cite{nes13} with an additional cost proportional to the \cu{square root of the number of agents in the network in the worst case}.
	
	In general, one might be interested in finding a WB for the original Wasserstein distance with no regularization term. That is, to solve problem~\eqref{w_barycenter} with $\gamma=0$. The next theorem explains a choice of $\gamma$ that provides a convergence rate results with respect to the non-regularized optimal transport based on the iterates of Algorithm~\eqref{alg:main}.
	
	\vspace{-0.2cm}
	\begin{theorem}
		Let $\varepsilon>0$, and assume that  $\|\nabla \mathcal{W}^*_{\gamma,\mathtt{q}}(\tilde{\mathtt{y}})\|_2 \leq \dd{G} $ on a ball $B_R(0)$. Moreover, set $\gamma = \varepsilon/(4\pd{ m} \log n)$. 
		\dd{Then, it holds that after }
		
		\vspace{-0.3cm}
		{\small \begin{align*}
			N &\geq \sqrt{\frac{\pd{128} \dd{G}^2 \pd{m} \log n}{ \varepsilon^2} \frac{d_{\max}}{d_{\min}} }
			\end{align*}}
		
		\vspace{-0.2cm}
		\noindent iterations, the outputs of Algorithm~\ref{alg:main}, i.e. $\mathtt{p}^*_N = [(p_N^*)_1^T,\cdots,(p_N^*)_m^T]^T$ and $\mathtt{y}^*_N = [(y_N^*)_1^T,\cdots,(y_N^*)_m^T]^T$ have the following properties
		{\small
			\begin{align*}
			\mathcal{W}_{0,\mathtt{q}}(\mathtt{p}^*_N)- \pd{\mathcal{W}_{0,\mathtt{q}}(\mathtt{p}^*)}
			\leq  \varepsilon \ \ \text{and} \ \ \|\sqrt{W}\mathtt{p}^*_N\|_2 \leq {\varepsilon}/{(\pd{2}R)}.
			\end{align*} 
		}
	\end{theorem}
	
	\begin{proof}
		\dd{Considering \pd{weak duality 
				$\mathcal{W}^*_{\gamma,\mathtt{q}}(\mathtt{y}^*_N)  \geq - \mathcal{W}_{\gamma,\mathtt{q}}(\mathtt{p}^*)$}
			and Theorem \ref{thm:main} with $\varepsilon/2$, to obtain
			{\small
				\begin{equation}
				\label{eq:duality}
				\mathcal{W}_{\gamma,\mathtt{q}}(\mathtt{p}_N^*) - \mathcal{W}_{\gamma,\mathtt{q}}(\mathtt{p}^*) \leq \varepsilon/2.
				\end{equation}
			}
			\pd{By the choice of $\gamma$,} for $\forall ~ i=1,...,m$, it holds that
			{\small
				\pd{ \begin{align*}
					\mathcal{W}_{\gamma,q_i}(p_i^*) - \mathcal{W}_{0,q_i}(p_i^*)  \leq {\varepsilon}/{(2m)}, \\
					\mathcal{W}_{\gamma,q_i}((p^*_N)_i) \geq \mathcal{W}_{0,q_i}((p^*_N)_i).
					\end{align*}}}
			\pd{Thus, the desired result follows from~\eqref{eq:duality},} and the two inequalities above.
		}
	\end{proof}

	\section{Numerical Experiments}\label{sec:numerical}
	
	In this section, we show two numerical experiments to validate the results of Theorem~\ref{thm:main}. We explore the problem of computing WB of univariate, discrete and truncated Gaussian densities and the computation of the WB of a subsample of $1000$ digit images from the MNIST dataset.  
	
	\subsection{Barycenter of Gaussian Distributions}
	
	Initially, we explore the computation of WB for sets of univariate, discretized and truncated Gaussian densities~\cite{Agueh2011}. We consider a network of agents where each agent $i$ holds a univariate, discretized and truncated Gaussian distribution, with mean  $\mu_i \in [-5,5]$, standard deviation $\sigma \in [0.1,2]$ and equally spaced support of $100$ points  in $[-5,5]$. The entropy regularization parameter is set to $\gamma = 0.1$. Figure~\ref{fig:results} shows the distance to optimality and the distance to consensus of Algorithm~\ref{alg:main} for the star, cycle, complete, and Erd\H{o}s-R\'enyi random graph graphs with a fixed size of $50$ nodes. Also,  Figure~\ref{fig:results} shows the scalability of the algorithm, i.e., the number of iterations required to reach an $\varepsilon$ accuracy in the distance to optimality and consensus for networks of increasing size.

\begin{figure}[htbp!]
	\vspace{-0.1cm}
	\centering
	\subfigure{
	\begin{tikzpicture}
		\node at (1,0.2)  {{\footnotesize {\color{black} \textbf{--}} Erd\H{o}s-R\'enyi}};
		\node at (0.84,0.5) {{\footnotesize {\color{red} \textbf{--}} Complete}};
		\node at (0.6,0.8) {{\footnotesize {\color{green} \textbf{--}} Cycle}};
        \node at (0.5,1.1) {{\footnotesize {\color{blue} \textbf{--}} Star}};
		\node at (1.5,2.2) {\small $e^*(\mathtt{y}^*_k)$};
		\begin{axis}[
        ticklabel style = {font=\tiny},
		width=4.7cm,height=3.5cm,scale=0.99,
		x label style={at={(axis description cs:0.5,0.15)},anchor=north,font=\small},
		y label style={at={(axis description cs:0.25,.5)},anchor=south},
		ymode = log,
        xmode = log,
		ymin = 1e-7, ymax=1e0,
		xmin = 1e1, xmax=1e5, 
		every axis plot/.append style={line width=1pt}],
		legend pos=south west;
		\addplot [black]     	table [x index=8,y index=0]{new_data_conv.dat};
		\addplot [red]     		table [x index=8,y index=2]{new_data_conv.dat};
        \addplot [green]     	table [x index=8,y index=4]{new_data_conv.dat};
		\addplot [blue]     	table [x index=8,y index=6]{new_data_conv.dat};
		\end{axis}
		\end{tikzpicture}
	\begin{tikzpicture}
	\node at (1.7,2.4) {\scriptsize $\argmin\limits_{k\geq 0} \{ k \mid e^*(\mathtt{y}^*_k) \leq \varepsilon_1\}$};
		\begin{axis}[
        ticklabel style = {font=\tiny},
		width=5.0cm,height=3.5cm,scale=0.99,
		x label style={at={(axis description cs:0.5,0.15)},anchor=north,font=\small},
		ymin = 1e0, ymax=0.06e7,
		xmin = 1, xmax=50, 
		every axis plot/.append style={line width=1pt}],
		legend pos=south west;
		\addplot [black]     	table [x index=8,y index=4]{new_data.dat};
		\addplot [red]     		table [x index=8,y index=5]{new_data.dat};
        \addplot [green]     	table [x index=8,y index=6]{new_data.dat};
		\addplot [blue]     	table [x index=8,y index=7]{new_data.dat};
		\end{axis}
		\end{tikzpicture}
	}
	\\
	\vspace{-0.5cm}
    	\subfigure{
    	\begin{tikzpicture}
		\node at (1.6,2.2) {\small $\|\sqrt{W}\mathtt{p}^*_k\|_2$};
		\begin{axis}[
        ticklabel style = {font=\tiny},
		width=4.7cm,height=3.5cm,scale=0.99,
		x label style={at={(axis description cs:0.5,0.15)},anchor=north,font=\small},
		y label style={at={(axis description cs:0.25,.5)},anchor=south},
		xlabel={Iterations},
		ymode = log,
        xmode = log,
		ymin = 1e-7, ymax=1e0,
		xmin = 1e1, xmax=1e5, 
		every axis plot/.append style={line width=1pt}],
		legend pos=south west;
		\addplot [black]     	table [x index=8,y index=1]{new_data_conv.dat};
		\addplot [red]     		table [x index=8,y index=3]{new_data_conv.dat};
        \addplot [green]     	table [x index=8,y index=5]{new_data_conv.dat};
		\addplot [blue]     	table [x index=8,y index=7]{new_data_conv.dat};
		\end{axis}
		\end{tikzpicture}
        \begin{tikzpicture}
	\node at (1.7,2.4) {\scriptsize $\argmin\limits_{k\geq 0} \{ k \mid \|\sqrt{W}\mathtt{p}^*_k\|_2 \leq \varepsilon_2\}$};
	\begin{axis}[
    ticklabel style = {font=\tiny},
	width=5.0cm,height=3.5cm,scale=0.99,
			x label style={at={(axis description cs:0.5,0.15)},anchor=north,font=\small},
			xlabel={Number of Agents},
	ymin = 1e1, ymax=0.35e6,
	xmin = 1, xmax=50, 
	every axis plot/.append style={line width=1pt}],
	legend pos=south west;
		\addplot [black]     	table [x index=8,y index=0]{new_data.dat};
		\addplot [red]     		table [x index=8,y index=1]{new_data.dat};
        \addplot [green]     	table [x index=8,y index=2]{new_data.dat};
		\addplot [blue]     	table [x index=8,y index=3]{new_data.dat};
	\end{axis}
		\end{tikzpicture}
	}
\vspace{-0.3cm}
	\caption{Optimality and Scalability for various graphs. $e^*(\mathtt{y}^*_k) = (\mathcal{W}^*_{\gamma,\mathtt{q}}(\mathtt{y}^*_k) - \mathcal{W}^*_{\gamma,\mathtt{q}}(\mathtt{y}^*))/ (\mathcal{W}^*_{\gamma,\mathtt{q}}(\mathtt{y}^*_0) - \mathcal{W}^*_{\gamma,\mathtt{q}}(\mathtt{y}^*))$, $\varepsilon_1=1\cdot 10^{-8}$ and $\varepsilon_2=1\cdot 10^{-6}$.}
	\label{fig:results}
	\vspace{-0.5cm}
\end{figure}
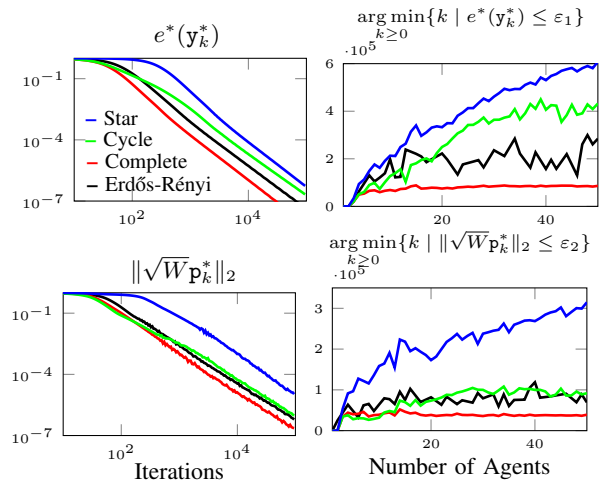

	\subsection{MNIST Dataset}
	
	We randomly sample $1000$ images for each digit of the MNIST dataset~\cite{LeCun1998,LeCun1998a}. Each image has $28\times28$ pixels and is scaled uniformly at random between $0.5$ and $2$ of its size and randomly located on a larger $56\times56$ blank image. The pixel values of the image are normalized to add up to $1$. We assign one sample from each digit to each agent on a group of $1000$ agents, and the objective is to jointly compute the WB for each digit of the $1000$ samples present in the network. \dd{Each agent owns only one image, and these images are different. In total the number of images assigned to each agent is equal to the number of digits.} The agents are connected over an Erd\H{o}s-R\'enyi random graph with $1000$ nodes and connectivity parameter $4/1000$. The entropy regularization parameter is set to $\gamma = 0.01$. Figure~\ref{fig:mnist} shows the local barycenter of the $9$ digits for a subset of $3$ agents in the network for $0$, $60$, and $300$ iterations. As the number of iteration increases, all agents converge to a common image, namely, the WB of the images held by the agents in the network.
	
	\vspace{-0.1cm}
	\begin{figure}[hth!]
		\centering
		\begin{tikzpicture}
		\node[rotate=90] at (0,0) { $\ \ \ \ N=0$};
		\end{tikzpicture}
		\fbox{\includegraphics[trim={1.3cm 1cm 1.1cm 5.76cm},clip,width=0.35\textwidth]{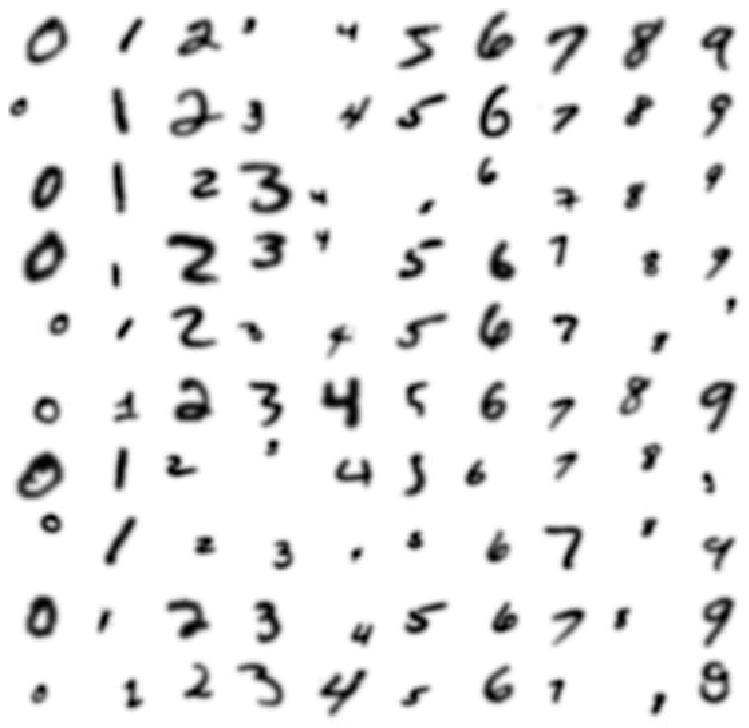}}
		\\
		\begin{tikzpicture}
		\node[rotate=90] at (0,0) { $\ \ \  N=60$};
		\end{tikzpicture}
		\fbox{\includegraphics[trim={1.3cm 1cm 1.1cm 5.76cm},clip,width=0.35\textwidth]{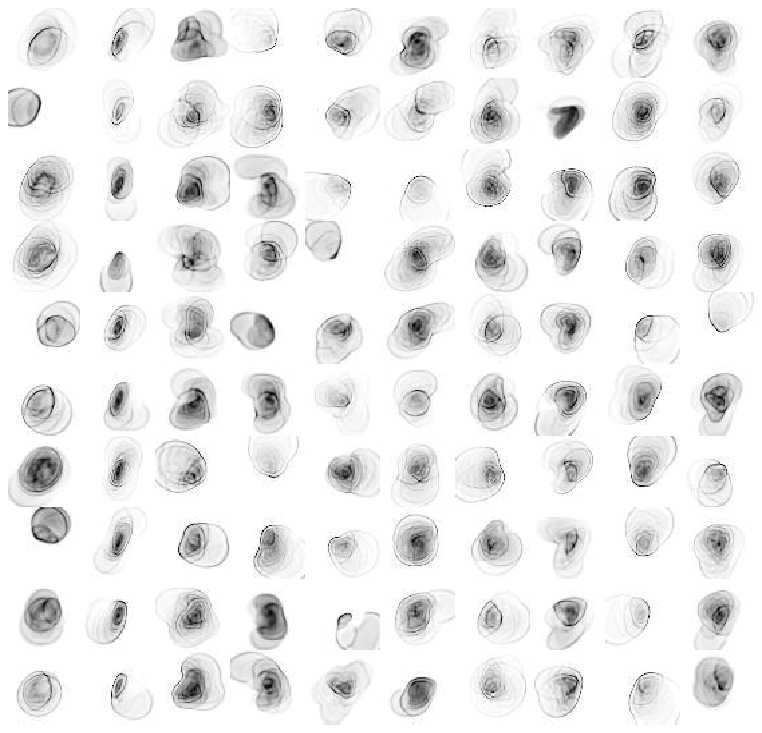}}
		\\
		\begin{tikzpicture}
		\node[rotate=90] at (0,0) { $\ \  N=300$};
		\end{tikzpicture}
		\fbox{\includegraphics[trim={1.3cm 1cm 1.1cm 5.76cm},clip,width=0.35\textwidth]{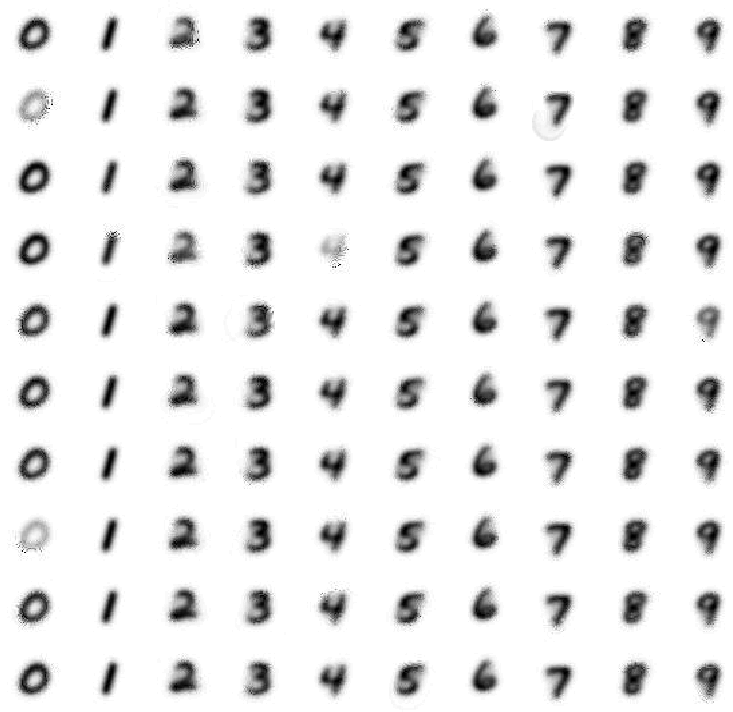}}
		\caption{Local WB of the digits of the MNIST dataset for a subset of $3$ agents out of $1000$ on an Erd\H{o}s-R\'enyi random graph. A video of the evolution of the local barycenters for $10$ agents is available at \url{http://bit.ly/2t9fn0Y}.}
		\label{fig:mnist}
		\vspace{-0.3cm}
	\end{figure}

	\section{Discussion and Conclusions}\label{sec:conclusions}

	We developed a novel algorithm for the distributed computation of WB over networks where a group of agents connected over a network and each agent holds some local probability distribution with finite support. Our results provably guarantee that all agents in the network will converge to the WB of all distribution held by the agents in the network. We provide an explicit and non-asymptotic convergence rate of the order $O(1/k^2)$ with an additional cost proportional to the ratio between the maximum and minimum degree among the nodes in the graph over which the agents exchange information.   
	
	The case where spectral information of the network is not available or when the graphs are directed or change with time require further study. The use of second-order information can also be exploited to get better performance. Also, recently proposed stochastic approaches can provide more efficient algorithms~\cite{Claici2018,Staib2017,dvurechensky2018decentralize,rogozin2018optimal}.

	\bibliographystyle{IEEEtran}
	\bibliography{IEEEfull,wass,opt_dec2}

\end{document}